%% file: main.tex
\PassOptionsToPackage{dvipsnames}{xcolor} 

\documentclass[11pt,letterpaper]{amsart}
\usepackage[a4paper, left=1in, right=1in, top=1.1in, bottom=1.1in]{geometry}

\usepackage[english]{babel}
\usepackage[utf8]{inputenc}
\usepackage{amsmath}
\usepackage{graphicx}
\usepackage{amssymb}
\usepackage{stmaryrd}
\usepackage{amsthm}
\usepackage{tikz-cd}
\usepackage{tikz}
\usetikzlibrary{arrows.meta,calc}
\usepackage{mathrsfs}
\usepackage{csquotes}
\usepackage[colorinlistoftodos]{todonotes}
\usepackage{yfonts}
\usepackage{ dsfont }
\usepackage{xcolor}
\usepackage{multicol}
\usepackage{indentfirst}
\usepackage[hyphens]{url}
\usepackage{hyperref}
\usepackage{thmtools}
\usepackage[hyphenbreaks]{breakurl}
\usepackage{subfiles}
\usepackage{longtable}
\usepackage{float}
\usepackage{subcaption}

\AtBeginDocument{%
  \captionsetup[subfigure]{labelformat=parens}%
}

\definecolor{tocblue}{RGB}{0,51,153}
\colorlet{tocblueLight}{tocblue!90!white} 
\definecolor{mrred}{HTML}{B33333} 
 \usepackage[
    backend=biber,
    style=alphabetic, maxbibnames=99, url=false
  ]{biblatex}
\usepackage{xurl}
\hypersetup{
    colorlinks,
    linkcolor={mrred},
    citecolor={tocblueLight},
    urlcolor={RoyalPurple}
}
\hypersetup{breaklinks=true}

\usepackage[inline]{enumitem}   
\makeatletter

\newcommand{\inlineitem}[1][]{%
\ifnum\enit@type=\tw@
    {\descriptionlabel{#1}}
  \hspace{\labelsep}%
\else
  \ifnum\enit@type=\z@
       \refstepcounter{\@listctr}\fi
    \quad\@itemlabel\hspace{\labelsep}%
\fi}
\makeatother
\allowdisplaybreaks

 \usepackage[
    backend=biber,
    style=alphabetic, maxbibnames=99, url=false
  ]{biblatex}
\usepackage{cleveref}

\title[Extremal Effective Cycles and Nef Line Bundles on \(\M{g}{n}\)]{Extremal Effective Cycles and Nef Line Bundles on \(\M{g}{n}\)}

\author{Daebeom Choi}
\address{Department of Mathematics\\
    University of Pennsylvania\\
    Philadelphia, PA 19104-6395}
\email{dbchoi@sas.upenn.edu}

\date{\today}
\theoremstyle{definition}
\newtheorem{thm}{Theorem}[section]

\newtheorem{lem}[thm]{Lemma}

\newtheorem{prop}[thm]{Proposition}

\newtheorem{defn}[thm]{Definition}

\newtheorem{rmk}[thm]{Remark}



\crefname{thm}{Theorem}{Theorems}
\crefname{lem}{Lemma}{Lemmas}
\crefname{prob}{Problem}{Problems}
\crefname{prop}{Proposition}{Propositions}
\crefname{qes}{Question}{Questions}
\crefname{defn}{Definition}{Definitions}
\crefname{eg}{Example}{Examples}
\crefname{ex}{Exercise}{Exercises}
\crefname{conj}{Conjecture}{Conjectures}
\crefname{defthm}{Definition-Theorem}{Definition-Theorems}
\crefname{cor}{Corollary}{Corollaries}
\crefname{claim}{Claim}{Claims}
\crefname{rmk}{Remark}{Remarks}
\crefname{section}{Section}{Sections}

\newcommand{\R}{\mathbb{R}}

\newcommand{\Z}{\mathbb{Z}}
\newcommand{\N}{\mathbb{N}}
\newcommand{\Q}{\mathbb{Q}}

\newcommand{\HH}{\text{H}}

\newcommand{\M}[2]{\overline{\rm{M}}_{#1, #2}}

\newcommand{\Mg}[1]{\overline{\rm{M}}_{#1}}

\newcommand{\Eff}{\text{Eff}}
\newcommand{\pEff}{\overline{\text{Eff}}}
\newcommand{\nc}{\mathrm{N}}

\DeclareSymbolFont{yhlargesymbols}{OMX}{yhex}{m}{n}
\DeclareMathAccent{\widetriangle}{\mathord}{yhlargesymbols}{"E6} 

\begin{document}
    
\maketitle

\begin{abstract}
    There has been a growing body of work devoted to the study of effective cones of codimension-\(k\) cycles \(\Eff^k(\M{g}{n})\) on \(\M{g}{n}\), the moduli space of \(n\) pointed stable curves of genus \(g\). In this paper, we remove the genus-dependence present in previous bounds on the number of marked points, and prove the following results: (1) \(\Eff^{k}(\M{g}{n})\) has infinitely many extremal rays for \(k\ge 2\), \(g\ge 3\) and \(n\ge 2k-2\), and (2) \(\Eff^k(\M{g}{n})\) is non-polyhedral for  \(k\ge 2\), \(g\ge 1\) and \(n\ge k+5\). Moreover, we show that (3) every rational tails boundary stratum spans an extremal ray. Our method refines that of Chen and Coskun by extending arguments based on morphisms, or equivalently semiample divisors, to a setting that also allows for the use of nef divisors. Certain non-semiample nef divisors on \(\M{g}{n}\), namely so-called semigroup kappa divisors of a particular kind, play a crucial role.
\end{abstract}

\section{Introduction}\label{sec:intro}

\subfile{sections/intro}

\section{Preliminaries}\label{sec:prelim}

\subfile{sections/prelim}

\section{Lemmas on extremality of effective cycles}\label{sec:lemmas}

\subfile{sections/lemmas}

\section{Semigroup kappa divisors}\label{sec:cont}

\subfile{sections/cont}

\section{Injectivity of pushforwards}\label{sec:inj}

\subfile{sections/inj}

\section{Extremal effective cycles on \texorpdfstring{\( \M{g}{n} \)}{TEXT}}\label{sec:ext}

\subfile{sections/ext}

\printbibliography

\end{document}

%% file: sections/intro.tex
The moduli space of stable genus \(g\) curves with \(n\) marked points, \(\M{g}{n}\), is a central object in algebraic geometry. One question about \(\M{g}{n}\) that has attracted considerable attention is to describe its cone of effective divisors \(\Eff^1(\M{g}{n})\), because of its relation to the Kodaira dimension and birational models of \(\M{g}{n}\) (see, e.g. \cite{HM82, EH87, HT02, Ver02, FP05, CC14, Opi16, Mul17, Mul21, CLTU23, FJP23, Mul25}). There has been growing interest in investigating the geometry of a projective variety \(X\) through its effective (resp. pseudo-effective) cones of higher-codimension cycles \(\Eff^k(X)\) (resp. \(\pEff^k(X)\), see \cref{subsec:notation}), e.g. \cite{OJV13, FL16, FL17, FL17b}. Correspondingly, there have been a number of papers devoted to the study of such higher-codimension effective cycles on \(\M{g}{n}\), e.g. \cite{CC15, Sch15, DN16, Mul20, Bla22, Cho25b}.

As explained in \cite{Bla22}, there are two main types of results concerning cones of (pseudo) effective cycles on \(\M{g}{n}\). The first is to determine whether they are polyhedral. The second is to prove that certain special strata in \(\M{g}{n}\) generate extremal rays in \(\Eff^k(\M{g}{n})\) or \(\pEff^k(\M{g}{n})\). In this work, we obtain results of both sorts. 

First, in \cref{thm:main theorem 1}, we identify new cones of effective codimension-\(k\) cycles on \(\M{g}{n}\) that are not polyhedral. Compared to previous results, which we will review, the main advance of \cref{thm:main theorem 1} is that the lower bound on \(n\) is independent of \(g\). 

\begin{thm}\label{thm:main theorem 1}
    \begin{enumerate}
        \item For \(k\ge 2\), \(g\ge 3\), and \( n\ge 2k-2 \), there are infinitely many extremal effective codimension \(k\) cycles on \(\M{g}{n}\), i.e. \(\Eff^k(\M{g}{n})\) has infinitely many extremal rays.
        \item For \(k\ge 2\) and \(g\ge 1\), \(\Eff^k(\M{g}{n})\) is not polyhedral for \( n\ge k+5 \).
    \end{enumerate}
\end{thm}

Second, extending \cite{Bla22}, we obtain new examples of extremal boundary strata.

\begin{thm}\label{thm:main theorem 2}
    Any rational tails boundary stratum is extremal.
\end{thm}

The common new ingredient in the proofs of \cref{thm:main theorem 1} and \cref{thm:main theorem 2} is \cref{lem:contracting subvariety}, which allows us to prove the extremality of the pushforward of an extremal effective cycle. It extends the argument of Chen and Coskun \cite[Proposition 2.1]{CC15}, which uses morphisms, or equivalently, their associated semiample divisors, to a numerical argument using the exceptional locus of nef divisors. The nef divisors needed for these arguments are supplied by the semigroup kappa divisors in \cref{prop:semigroup Kappa} (see also \cite[Section 7]{Cho25b}). Since the particular nef divisors used in the proofs are not semiample (cf. \cref{rmk:semiample}), \cref{lem:contracting subvariety} is essential.

As some of these are used in the proofs, we now briefly catalog other results of these two types: the polyhedrality of \(\Eff^k(\M{g}{n})\) and \(\pEff^k(\M{g}{n})\), and the extremality of boundary strata. 

The following cones of cycles on \(\M{g}{n}\) are known to have infinitely many extremal rays:

\begin{itemize}
    \item \( \Eff^1(\M{1}{n}) \) for \(n\ge 3\) \cite{CC14}.
    \item \( \Eff^2(\M{1}{n}) \) for \(n\ge 5\) and \( \Eff^2(\M{2}{n}) \) for \(n\ge 2\) \cite{CC15}.
    \item \( \Eff^1(\M{g}{n}) \) for \(g\ge 2\), \(n\ge g+1\) \cite{Mul17}.
    \item \( \Eff^k(\M{1}{n}) \) for \(n\ge k+2\), \( \Eff^k(\M{g}{n}) \) for \(g\ge 2\), \(k\le g\), and \(n\ge g+k\), and \( \Eff^2(\M{g}{n}) \) for \(n\ge g+1\) and \(g\ge 3\) \cite{Mul20}.
\end{itemize}

The following pseudo-effective cones are known not to be polyhedral:

\begin{itemize}
    \item \( \pEff^1(\M{g}{n}) \) for \(g,n\ge 2\) \cite{Mul21}.
    \item \( \pEff^1(\M{0}{n}) \) for \( n\ge 8 \) \cite{Mul25}, improving the result of \cite{CLTU23}, which holds for \(n\ge 10\).
\end{itemize}

Given these results, we can see the improvement of \cref{thm:main theorem 1} is that the lower bound on \(n\) no longer depends on \(g\), but only on \(k\). 

To clarify, we next explain three kinds of statements, all of which appear in the main theorems here and in previous work regarding non-polyhedrality of effective cones. Let \(C\) be a full-dimensional cone in a Euclidean space, possibly non-closed. In our case, this will be \(\Eff^k(X)\) for some projective variety \(X\). Here are three statements:

\begin{enumerate}
    \item \(C\) has infinitely many extremal rays.
    \item The closure \( \overline{C} \) of \(C\) is not polyhedral.
    \item \(C\) is not polyhedral.
\end{enumerate}

Each of (1) and (2) implies (3), but the other implications do not hold in general. For example, the interior of a circular (resp. polyhedral) cone satisfies (2) but not (1) (resp. (3) but not (1) or (2)). If we take the union of the interior of a cube and a closed disk on one of the six faces of the cube, then this set is convex; if we let \(C\) be the \(4\)-dimensional cone over it, then this gives a counterexample to \( (1)\Rightarrow (2) \). Although the author is unaware of whether such a counterexample can occur for \(C=\Eff^k(X)\), we attempt here to be precise and distinguish these three statements.

We now summarize the known results on extremal boundary strata of $\M{g}{n}$.

\begin{itemize}
    \item Boundary divisors: See \cite{Rul01, CC15, Che18, Bla22}.
    \item  Codimension \(2\) boundary strata of \(\M{0}{n}\) and \(\Mg{g}\), and certain higher-codimension boundary strata \cite[Theorem 5.6, Theorem 6.2]{CC15}.
    \item All boundary strata of \( \M{0}{n} \) \cite{Sch15, Bla22}.
    \item Certain \textit{rational tails boundary strata} (cf. \cref{defn:rational tail}); more specifically, such strata without trivalent vertices of genus \(0\) and \textit{pinwheel strata} \cite[Proposition 4.10]{Bla22}.
    \item Certain 1-dimensional boundary strata \cite[Theorem 1.3, Theorem 1.4]{Cho25b}.
\end{itemize}

In \cref{thm:main theorem 2}, we extend the work of \cite{Bla22} by showing that any rational tails boundary stratum is extremal.

We next describe our contributions, and the method used to extend the work of Chen and Coskun \cite[Proposition 2.1]{CC15}. Following the strategy of \cite{CC15, Bla22}, we prove the extremality of the pushforward of certain extremal effective cycles. Therefore, the first step is to identify suitable extremal effective cycles to start with. For this step, we do not construct such cycles ourselves; instead, we rely on cycles constructed in previous work, all of which were listed above. In this sense, the paper gives a general procedure for producing new extremal effective cycles from previously known ones. For example, if we know that there are infinitely many extremal effective divisors on \(\M{0}{n}\) for sufficiently large \(n\), then we can strengthen part (2) of \cref{thm:main theorem 1} to a stronger statement asserting the infinitude of extremal effective cycles.

In particular, we take pushforwards of certain extremal effective cycles along \(\iota:\overline{\rm{M}}\to \M{g}{n}\), where \(\overline{\rm{M}}\) is a certain product of moduli spaces of curves and \(\iota\) is the corresponding clutching map. Using \cref{lem:extremal product everything} and \cref{lem:extremal product point}, we construct extremal effective cycles on \(\overline{\rm{M}}\) from those in the previous paragraph. To prove the extremality of the pushforward \(\iota_\ast : \nc^k(\overline{\rm{M}})\to \nc^{k+d}(\M{g}{n})\) of such cycles, we use \cref{lem:main strategy}, which already appears in \cite{CC15, Bla22}. 

To apply \cref{lem:main strategy}, we need two ingredients: (1) injectivity of the pushforward, and (2) a confinement condition on the summands appearing in the pushforward. More precisely, the second condition says that if \(Z\) is one of the extremal effective cycles on \(\overline{\rm{M}}\) mentioned above, and if \(\iota_\ast [Z]=\sum a_j[Z_j]\) with \(a_j>0\), then \(Z_j=\iota(Z'_j)\) for some subvariety \(Z'_j\subseteq \overline{\rm{M}}\) for each \(j\). In other words, all summands of \(\iota_\ast [Z]\) are confined to the image of \(\iota\). To establish this, \cite{CC15, Bla22} use morphisms, such as Hassett's weighted moduli spaces of curves \cite{Has03} and the moduli space of pseudostable curves \cite{HH09}, together with \cite[Proposition 2.1]{CC15}.

Here, we prove \cref{lem:contracting subvariety}, which strengthens \cite[Proposition 2.1]{CC15}. The argument in \cite[Proposition 2.1]{CC15} requires a morphism, hence applies only when one has a semiample divisor. In \cref{lem:contracting subvariety}, we extend this to a setting that also allows the use of nef divisors. The key advantage is that proving nefness is often much easier than proving semiampleness, and some of the nef divisors used in this paper are in fact not semiample (cf. \cref{rmk:semiample}). Thus, \cref{lem:contracting subvariety} is a crucial improvement that makes the arguments of this paper possible.

The nef divisors we use are semigroup kappa divisors, introduced in \cite[Section 7]{Cho25b}. They are easy to construct, and the structure of semigroup kappa divisors makes it easier to keep track of their behavior on the boundary. For the main theorems, we only need a very special class of semigroup kappa divisors, described in \cref{prop:semigroup Kappa}, but other semigroup kappa divisors are also useful for extremality problems, as we will see in \cref{prop:codim 2}. This suggests that semigroup kappa divisors may provide a useful source of nef divisors for studying extremality problems beyond the cases considered in this paper.

The final step is the injectivity of certain pushforwards, established in \cite[Proposition 4.2]{Bla22} and \cref{prop:N1 pushforward injective}. Unlike similar results in \cite{CC15, Bla22}, we convert the injectivity statement into a surjectivity statement on the dual groups, and then prove surjectivity using the combinatorics of dual graphs, which will be reviewed in \cref{sec:prelim}. This part is the bottleneck of the current method: if we can improve \cref{prop:N1 pushforward injective}, then we can further improve \cref{thm:main theorem 1} (cf. \cref{rmk:inj} (2)). However, as the codimension grows, the combinatorics of dual graphs become more complicated and further obstructions arise (cf. \cref{rmk:inj} (3)). 

The paper is organized as follows. In \cref{sec:prelim}, we review the basics of dual graphs and intersection theory. In \cref{sec:lemmas}, we present the lemmas used to transfer extremality, following \cite{CC15, Bla22}. In \cref{sec:cont}, we introduce a simple class of semigroup kappa divisors, which will be used to prove extremality. In \cref{sec:inj}, we prove and discuss the injectivity of pushforward maps between numerical Chow groups. Finally, in \cref{sec:ext}, we prove the main theorems.

\section*{Acknowledgments}

The author would like to thank Angela Gibney for her helpful discussions and continued support. We are grateful to Jenia Tevelev for answering our questions, and to Dawei Chen for helpful comments on an earlier version of this paper. The author was partially supported by the Korea Foundation for Advanced Studies Overseas PhD Scholarship Program and Simons Dissertation Fellowship SFI-MPS-SDF-00014683.

%% file: sections/prelim.tex
\subsection{Notations and conventions}\label{subsec:notation}

Throughout, the Picard group $\rm{Pic}(X)$, the cone of nef divisors $\rm{Nef}(X)$, numerical Chow groups $\nc_d(X), \nc^d(X)$, curve classes $[C]$, and divisors/line bundles will be considered over $\Q$. Thus, unless otherwise stated, these terms refer to their $\Q$-coefficient versions, such as the $\Q$-Picard group, $\Q$-divisors, and so on. \(\Eff^d(X)\) is the subcone of \(\nc^d(X)_\R\) generated by effective cycle classes, and \(\pEff^d(X)\) is the closure of \(\Eff^d(X)\). When we say that a codimension \(d\) closed subvariety \(Z\subseteq X\), or its class \([Z]\), is extremal, we mean that \([Z]\) spans an extremal ray of \(\Eff^d(X)\). We include integrality as part of the definition of a variety.

We denote by $[n]$ the set $\{1, 2, \dotsc, n\}$, and by $\pi_S$ the projection map $\M{g}{n} \to \M{g}{S^c}$ that forgets the marked points indexed by $S \subseteq [n]$. We will denote a stable curve with \(n\) marked points by \((C, p_1,\cdots, p_n)\), but for simplicity we will sometimes write \((C, p_\bullet)\), and we will denote the \(i\)th marked point by \(i\) instead of \(p_i\).

\subsection{Dual graphs and boundary strata}
In this subsection, we recall some basic facts about dual graphs, since some of the proofs in this paper involve the combinatorics of dual graphs. We refer to \cite{Sch20} for a detailed explanation.

\begin{defn}\label{defn:dual graph}
    A \textbf{stable graph} is a tuple \( \Gamma=(V, H, E, L, g,v,\iota, l, n) \) such that 
    \begin{itemize}
        \item \(V=V(\Gamma)\), the finite set of \textbf{vertices}, and \(g:V\to \mathbb{Z}_{\ge 0}\) associates a \textbf{genus} to a vertex.
        \item \( H=H(\Gamma) \) is a set of \textbf{half-edges}, with a function \(v:H\to V\) indicating the vertex to which a half-edge is attached, and an involution \(\iota:H\to H\).
        \item  \(L=L(\Gamma)\) is a set of \textbf{legs}, the subset of fixed points of \(\iota\) with a bijection \(l:L(\Gamma)\to [n] \).
        \item \(E=E(\Gamma)\), the set of \textbf{edges}, is the quotient of \(H\setminus L\) by \(\iota\). 
    \end{itemize}
    which satisfies the following conditions:
    \begin{enumerate}
        \item The underlying graph \((V, E)\) is connected.
        \item If \(g(v)=0\) (resp. \(g(v)=1\)), then there are at least \(3\) (resp. \(1\)) half-edges incident to \(v\).
    \end{enumerate}
\end{defn}

For any dual graph \(\Gamma\), the associated graph \((V,E)\) has a natural structure of a \(1\)-dimensional CW complex. We denote its first Betti number by \(h^1(V(\Gamma), E(\Gamma))\). The \textbf{genus} of a dual graph \(\Gamma\) is defined by 
\[ g(\Gamma)=\sum_{v\in V(\Gamma)}g(v)+h^1(V(\Gamma), E(\Gamma)). \]
Moreover, the \textbf{number of marked points} of \(\Gamma\) is \(n(\Gamma)=|L|=n\).

Dual graphs parametrize boundary strata of \(\M{g}{n}\). For a stable curve \((C, p_\bullet)\), the corresponding dual graph \(\Gamma_C\) is obtained in the following way:
\begin{itemize}
    \item \( V(\Gamma_C) \) is the set of irreducible components of \(C\), and \(g\) is the genus of each component.
    \item \(L(\Gamma_C)\) is the set of marked points of \(C\), and \(l\) maps \(p_i\) to \(i\).
    \item \(E(\Gamma_C)\) is the set of nodes of \(C\) connecting the vertices corresponding to the irreducible components containing the node.
    \item \( H(\Gamma_C) \) is the disjoint union of two copies of \(E(\Gamma_C)\) and \(L(\Gamma_C)\). The map \(v\) assigns to each marking in \(L(\Gamma_C)\) the component containing it, and assigns to each node in \(E(\Gamma_C)\) the two possibly identical components containing that node.
\end{itemize}
It is straightforward to verify that \(g(\Gamma_C)=g\) and \(n(\Gamma_C)=n\). For any dual graph \(\Gamma\) with \(g=g(\Gamma)\) and \(n=n(\Gamma)\), let \(\rm{M}_{\Gamma}\) be the locus in \(\M{g}{n}\) parametrizing curves \((C, p_\bullet)\) with \(\Gamma_C\simeq\Gamma\), and let \(\Delta_{\Gamma}\) be its closure in \(\M{g}{n}\). Then \(\Delta_{\Gamma}\) is a boundary stratum of \(\M{g}{n}\). This defines a one-to-one correspondence between dual graphs and boundary strata. The codimension of \(\Delta_{\Gamma}\) is equal to the number of edges of \(\Gamma\).

The \(1\)-dimensional boundary strata are especially interesting for various reasons. Such strata are called \textbf{F-curves}. For a detailed treatment of F-curves, we refer to \cite[Section 2]{GKM02}. F-curves are classified into six types, and following the notation of \cite{Cho25b}, we will denote these six types of F-curves by \(F_1\), \(F_2\), \(F_3^{g_1}(I_1)\), \(F_4^{g_1}(I_1)\), \(F_5^{g_1,g_2}(I_1,I_2)\), and \(F_6^{g_1,g_2,g_3,g_4}(I_1,I_2,I_3,I_4)\). The corresponding dual graphs give explicit representatives of these six types; we refer to \cite[Section 2]{Cho25b} for the full definition.

\subsection{Intersection Theory and Positivity}

Here, we review some basics of intersection theory on varieties and the positivity of cycles. Let \(X\) be a proper variety, and let \(A_k(X)\) (resp. \(\nc_k(X)\)) denote the Chow group of \(k\)-cycles modulo rational (resp. numerical) equivalence, with \(\Q\)-coefficients (for details, see, e.g. \cite[Chapter 19]{Ful98}). In \(\nc_k(X)_{\R}\), let \(\Eff_k(X)\), the \textbf{cone of effective \(k\)-cycles}, be the cone generated by effective \(k\)-cycles, and let \(\pEff_k(X)\), the \textbf{cone of pseudo-effective \(k\)-cycles}, be its closure. For a detailed discussion, we refer to \cite[Section 2]{FL17}. The following proposition summarizes basic properties of \(\nc_k(X)\) and these cones.

\begin{prop}\label{prop:numerical chow facts}
Let \(X\) and \(Y\) be projective varieties and \(f:X\to Y\) be a surjective morphism.

\begin{enumerate}
    \item For every \(X\), the vector space \(\nc_k(X)\) is finite-dimensional.

    \item If \(L\) is a line bundle on \(X\) and \( \alpha\in \nc_k(X) \), then \(c_1(L)\cap \alpha\) is well-defined in \(\nc_{k-1}(X)\).

    \item If \(L\) is semiample and \( \alpha=[V]\in \Eff_k(X) \) is an effective cycle, then \(c_1(L)\cap \alpha\in \Eff_{k-1}(X)\). If \(L\) is nef and \( \alpha\in \pEff_k(X) \), then \(c_1(L)\cap \alpha\in \pEff_{k-1}(X)\). In general, for \(l\le k\), we have \(c_1(L)^l\cap \alpha\in \pEff_{k-l}(X)\) for any nef line bundle \(L\).

    \item \(\pEff_k(X)\) is a salient cone, i.e. \( \pEff_k(X)\cap (-\pEff_k(X))=\{0\} \).

    \item \(f_\ast(\Eff_k(X))=\Eff_k(Y) \) and \(f_\ast(\pEff_k(X))=\pEff_k(Y) \).
\end{enumerate}
\end{prop}

\begin{proof}
    (1) This follows from \cite[Example 19.1.4]{Ful98}.

    (2) This is the Chern class operation; see \cite[Chapters 2,19]{Ful98} and \cite[Remark 2.4]{FL17}.

    (3) The second assertion follows from the first one by taking limits, since \(\pEff_{k-1}(X)\) is the closure of \(\Eff_{k-1}(X)\). The first assertion follows from \cite[Proposition 2.5(c)]{Ful98} applied to the closed embedding \(\iota:V\to X\), since \(c_1(L|_V)\) defines an effective \(\Q\)-Cartier divisor for any closed subvariety \(V\subseteq X\). The last assertion follows by repeated application of the second assertion.

    (4) This is \cite[Corollary 3.8]{FL17}.

    (5) The second assertion is \cite[Corollary 3.22]{FL17}. The first assertion follows from a standard argument: let \(Z\subseteq Y\) be a \(k\)-dimensional closed subvariety. It is enough to show that there exists a closed subvariety \(Z'\subseteq X\) of the same dimension such that \(f(Z')=Z\). Let \(\eta\) be the generic point of \(Z\), and consider the fiber \(X_{\eta}\) of \(f\). Take any closed point of \(X_\eta\). Its closure \(Z'\) in \(X\) is a desired subvariety.
\end{proof}

A line bundle \(L\) on a projective variety is \textbf{big} if \(\dim \HH^0(X, L^m)\) grows on the order of \(m^{\dim X}\) as \(m\to \infty\). For any line bundle \(L\) on \(X\), the \textbf{exceptional locus} of \(L\) is defined as 
\[ \mathbb{E}(L):=\bigcup_{\,V \subset X,\ \dim V>0,\ L|_V\ \text{not big}} V. \]

\begin{thm}\label{thm:big}
    Let \(L\) be a line bundle on \(X\).
    \begin{enumerate}
        \item If \(L \cong A+E\) with \(A\) an ample and \(E\) an effective \(\Q\)-Cartier divisor, then
        \[ \mathbb{E}(L)\subseteq \text{Supp}(E). \]
    
        \item If \(L\) is a nef line bundle, then \(L\) is big if and only if \(c_1(L)^{\dim X}\cap [X]\) is positive.
        
        \item If \(L\) is a nef line bundle, then \(\mathbb{E}(L)\) is a closed subvariety of \(X\). 

    \end{enumerate}
    
\end{thm}

\begin{proof}
    (1) Let \(V\subseteq X\) be a positive-dimensional closed subvariety such that \(L|_V\) is not big. If \(V\) is not contained in \(E\), then \(E|_V\) is again an effective divisor on \(V\), and \(A|_V\) is ample. Therefore, by \cite[Corollary 2.2.7]{La04a}, \(L|_V=A|_V+E|_V\) is big, a contradiction.

    (2) is \cite[Theorem 2.2.16]{La04a}, and (3) Follows from \cite[Theorem 1.4]{Bir17}.
\end{proof}

%% file: sections/lemmas.tex
In this section, we prove some lemmas regarding the extremality of effective cycles. Some of these lemmas are taken from \cite[Section 2]{CC15} and \cite[Section 2]{Bla22}. We recall the needed statements for the reader’s convenience. The following lemma will be our basic tool for generating extremal effective cycles. We say that \([Z]\in \Eff_k(X)\) is an \textbf{extremal effective cycle} if it spans an extremal ray of \(\Eff_k(X)\).

\begin{lem}(\cite[Proposition 2.5]{CC15}, \cite[Lemma 2.7]{Bla22})\label{lem:main strategy}
    Let \( \iota:W\to X \) be a morphism between projective varieties and \([Z]\in \Eff_k(W)\) be an extremal effective cycle. Assume
    \begin{enumerate}
        \item \(\iota_\ast: \nc_k(W)\to \nc_k(X) \) is injective, and
        \item If \( [\iota(Z)]=\sum_{j} a_j[Z_j]\text{ in }\nc_k(X)\) for some \(a_j>0\), then \(Z_j\subseteq \iota(W)\).
    \end{enumerate}
    Then \([\iota(Z)]\) spans an extremal ray of \(\Eff_k(X)\), i.e. is an extremal effective cycle.
\end{lem}

In this paper, \(\iota\) will be a clutching map between moduli spaces of curves. We need to check (1) of \cref{lem:main strategy} for such maps, which will be done in \cref{sec:inj}. For (2), we will use the following lemma, which is a generalization of \cite[Proposition 2.1]{CC15}.

\begin{lem}\label{lem:contracting subvariety}
    Let \(L\) be a nef line bundle on a projective variety \(X\). For any pseudo-effective \(k\)-cycle \(\alpha\in \pEff_k(X)\), define
    \[ e_L(\alpha):=k+1-\min \Bigl( \{k+1\}\cup \{ l\in \Z\cap [0,k] \ |\ c_1(L)^l\cap \alpha=0 \text{ in }\pEff_{k-l}(X)\subseteq \nc_{k-l}(X)_\R \}  \bigr)  \]
    Suppose that pseudo-effective \(k\)-cycles \(\alpha, \alpha_1, \cdots, \alpha_m\in \pEff_k(X)\) satisfy
    \[ \alpha=\sum_{j=1}^{m} a_j\alpha_j\text{ in }\nc^k(X)_\R \]
    for \( a_j>0 \). Then \( e_L(\alpha_j)\ge e_L(\alpha) \) for all \(1\le j\le m\). In particular, if \(k\)-dimensional subvarieties \(Z, Z_1, \cdots, Z_m\) satisfy
    \[ [Z]=\sum_{j=1}^m a_j[Z_j]\text{ in }\nc^k(X)_\R \]
    for \( a_j>0 \), then \( e_L(Z_j)\ge e_L(Z) \) for every \(j\). Moreover, if \(L|_Z\) is not big, then each \(Z_j\) is also contained in the exceptional locus \(\mathbb{E}(L)\) of \(L\).
\end{lem}

\begin{proof}
    Fix \(l\) such that \(c_1(L)^l\cap \alpha=0\). It suffices to show that \(c_1(L)^l\cap \alpha_j=0\). We have
    \[ c_1(L)^l\cap \alpha= \sum_{j=1}^m a_j(c_1(L)^l\cap \alpha_j)=0.\]
    Note that \( c_1(L)^l\cap \alpha_j\in \pEff_{k-l}(X) \) by \cref{prop:numerical chow facts} (3). Hence, by \cref{prop:numerical chow facts} (4), \(c_1(L)^l\cap \alpha_j=0\) for all \(1\le j\le m\). 

    Let \(\iota: Z \to X\) be the embedding. Then
    \[c_1(L)^l\cap [Z]=\iota_\ast (c_1(\iota^\ast L)^l\cap [Z]).\]
    Since \(c_1(\iota^\ast L)^{\dim Z}\cap [Z]\in N_{0}(Z)\simeq \Q\), \(c_1(L)^{\dim Z}\cap [Z]=0\) if and only if \(c_1(\iota^\ast L)^{\dim Z}\cap [Z]=0\). Therefore, \(e_L(Z)\ge 1\) if and only if \(L|_Z\) is not big. It follows that if \(L|_Z\) is not big, then \(L|_{Z_j}\) is also not big, and hence \(Z_j\subseteq \mathbb{E}(L)\).
\end{proof}

The following proposition implies that if \(L\) is a semiample divisor, then \cref{lem:contracting subvariety} specializes to \cite[Proposition 2.1]{CC15}.

\begin{prop}\label{prop:compatible}
    Let \(f:X\to Y\) be a morphism to a projective variety, let \(H\) be an ample line bundle on \(Y\), and \(L:=f^\ast H\). Then for any subvariety \(Z\subseteq X\),
    \[ e_L(Z)=e_f(Z) \]
    where \(e_f(Z):=\dim Z-\dim f(Z)\). 
\end{prop}

\begin{proof}
    Let \(\iota:Z\hookrightarrow X\) be the inclusion and \(h:=f\circ \iota\). Then, by \cite[Proposition 2.5(c)]{Ful98},
    \[ c_1(L)^l\cap [Z]=\iota_\ast (c_1(\iota^\ast L)^l\cap [Z])=\iota_\ast (c_1(h^\ast H)^l\cap [Z]) .\]
    Since \(c_1(h^\ast H)^l\cap [Z]\) is effective by \cref{prop:numerical chow facts} (3), we have \(c_1(L)^l\cap [Z]=0\) if and only if \(c_1(h^\ast H)^l\cap [Z]=0\). This is equivalent to \(l\ge \dim f(Z)+1\), which implies the assertion. 
\end{proof}

In \cite{Che18}, Dawei Chen asked for a generalization of \cite[Proposition 2.1]{CC15} to rational maps, whose corresponding divisor is effective. \cref{lem:contracting subvariety} may be viewed as a step toward this question, replacing morphisms by nef divisors. In \cref{rmk:semiample}, we will see that this is essential: we use a nef divisor for which the corresponding morphism does not exist in characteristic \(0\).

To apply \cref{lem:contracting subvariety}, we need a nef divisor \(L\) and a closed subvariety \(Z\). For \(L\), we will use the semigroup \(\kappa\) divisors defined in \cref{sec:cont}. To define \(Z\), we will use extremal effective cycles from previous work, modified using \cref{lem:extremal product everything} and \cref{lem:extremal product point} below.

\begin{lem}(\cite[Lemma 2.6]{Bla22})\label{lem:extremal product everything}
    Let \(X\) and \(Y\) be projective varieties. Let \( [Z]\in \Eff^d(X) \) be an extremal effective cycle. Then \( [Z\times Y]\in \Eff^d(X\times Y) \) is also extremal.  
\end{lem}

\begin{lem}\label{lem:extremal product point}
    Let \(X\) and \(Y\) be projective varieties. Let \( [Z]\in \Eff^d(X) \) be an extremal effective cycle. Then \( [Z\times y]\in \Eff^{d+\dim Y}(X\times Y) \) is also extremal for any point \(y\in Y\).  
\end{lem}

\begin{proof}[Proof of \cref{lem:extremal product point}]
    Let \( [E_1],\cdots, [E_m]\in \Eff^{d+\dim Y}(X\times Y) \) be effective cycles such that
    \begin{equation}\label{eqn:eff}
        [Z\times y]=\sum_{j=1}^m a_j[E_j]
    \end{equation}
    for some \(a_j>0\). Consider \( \pi_2:X\times Y\to Y \). Then, by \cite[Proposition 2.1]{CC15},
    \[ e_{\pi_2}(E_j)\ge e_{\pi_2}(Z\times y)=\dim Z, \]
    where \(e_{\pi_2}\) is as defined in \cite[Proposition 2.1]{CC15} or \cref{prop:compatible}. Since \(\dim E_j=\dim Z\), this means that \( \pi_2(E_j) \) is a point, i.e. \(E_j=Z_j\times y_j \) for some point \(y_j\in Y\) and a subvariety \(Z_j\subseteq X\). By taking the pushforward of \cref{eqn:eff} along \( \pi_1:X\times Y\to X \), we obtain
    \[ [Z]=\sum_{j=1}^m a_j[Z_j]. \]
    Since \(Z\) is extremal, \([Z_j]\) is proportional to \([Z]\), so \( [E_j] \) is also proportional to \( [Z\times y] \). Since we work modulo numerical equivalence, the classes of all closed points of \(Y\) coincide in \(\N_0(Y)\). Hence, \([Z\times y]\) is extremal.
\end{proof}

%% file: sections/cont.tex
In \cite{Cho25b}, the author defined \textbf{semigroup kappa divisors} and proved that they are always nef and, moreover, semiample in positive characteristic. The definition of semigroup kappa divisors is combinatorial, and we refer to \cite[Definition 7.1]{Cho25b}. They form a large class of nef divisors, but in this paper we need only a small subset of these divisors. Recall that
\[ \kappa=12\lambda+\psi-\delta \]
is an ample divisor on \(\M{g}{n}\) by \cite{Cor93}. 

\begin{prop}\label{prop:semigroup Kappa}
    Let \(\Delta_{i,I}\) be a boundary divisor satisfying the following condition: 
    \begin{equation}\label{eqn:cond}
        \text{If }I=\emptyset, \text{ then }2i>g\text{, and if }I=[n]\text{, then }2i<g. 
    \end{equation}
    Then \(\kappa+\delta_{i,I}\) is nef. 
\end{prop}

\(\kappa+\delta_{i,I}\) is a special case of semigroup kappa divisors, and a generalization of \cref{prop:semigroup Kappa} is proved in \cite[Theorem 7.5]{Cho25b}. We include the proof of \cref{prop:semigroup Kappa} for the convenience of the reader since in this case it is simple, and the proof illustrates how the corresponding morphisms look if they exist. Note that \(\Delta_{i, I}\simeq \M{i}{I+1}\times \M{g-i}{I^c+1}\), and define \(\pi_{1, s}\) and \(\pi_{2, s}\) by
\begin{equation}\label{eqn:proj}
\begin{aligned}
     &\pi_{1,s}:\M{i}{I+1}\times \M{g-i}{I^c+1}\to \M{i}{I+1}\to \M{i}{I},\\ &\pi_{2,s}:\M{i}{I+1}\times \M{g-i}{I^c+1}\to \M{g-i}{I^c+1}\to \M{g-i}{I^c},
\end{aligned}
\end{equation}
that is, we project onto each factor and then forget the node. If \(\M{i}{I+1}\) or \(\M{g-i}{I^c+1}\) is \(\M{0}{3}\) or \(\M{1}{1}\), then the corresponding map \(\pi_s\) is defined to be the constant map.

\begin{proof}
    Since \(\kappa\) is ample, the exceptional locus of \(\kappa+\delta_{i,I}\) is contained in \(\Delta_{i,I}\). Hence, it is enough to prove that the restriction of \(\kappa+\delta_{i,I}\) to \(\Delta_{i,I}\) is nef. We have
    \begin{equation}\label{eqn:Kappa}
        (\kappa+\delta_{i,I})|_{\Delta_{i,I}}=\pi_{1, s}^\ast \kappa + \pi_{2, s}^\ast \kappa
    \end{equation}
    by \cite[Lemma 1]{AC09} and \cref{eqn:cond}. Since \(\kappa\) is ample, this is nef. 
\end{proof}

In general, contractions corresponding to semigroup kappa divisors \cite[Definition 7.1]{Cho25b} can be described in a similar combinatorial way: they forget nodes of certain boundary strata. Although we will only use the divisors in \cref{prop:semigroup Kappa} for the main theorems of the paper, so the reader who is only interested in their proofs can safely skip the rest of this section, other semigroup kappa divisors are also useful, as the following example shows.

Choose \(g_i, I_i\) for \(1\le i\le 3\) such that \(g_1+g_2+g_3=g\), \(I_1 \sqcup I_2\sqcup I_3=[n]\), the \(I_i\)'s are nonempty and \(|I_i|\ge 3\) if \(g_i=0\). Let
\[ \Delta_{\Gamma_1}:=\Delta_{g_2, I_2}\cap \Delta_{g_3, I_3}, \Delta_{\Gamma_2}:=\Delta_{g_3, I_3}\cap \Delta_{g_1, I_1}, \Delta_{\Gamma_3}:=\Delta_{g_1, I_1}\cap \Delta_{g_2, I_2}. \]

\begin{prop}\label{prop:codim 2}
    Each \(\Delta_{\Gamma_i}\) spans an extremal ray in \(\Eff^2(\M{g}{n})\).     
\end{prop}

\begin{proof}
    The proof is similar to that of \cite[Theorem 5.3]{CC15}; however, they use morphisms, whereas we use nef divisors. Note that, by \cite[Example 7.7 (2)]{Cho25b},
    \[L:=\kappa+\delta_{g_1, I_1}+\delta_{g_2, I_2}+\delta_{g_3, I_3}\]
    is nef. Let 
    \[ \iota_1: \M{g_1}{I_1+2}\times \M{g_2}{I_2+1}\times \M{g_3}{I_3+1}\twoheadrightarrow \Delta_{\Gamma_1}\subseteq \M{g}{n} \]
    be the gluing map corresponding to \(\Delta_{\Gamma_1}\). Then, by direct computation, 
    \[ \iota_1^\ast L= \pi_{1,s}^\ast \kappa + \pi_{2,s}^\ast \kappa+\pi_{3,s}^\ast \kappa \]
    where \(\pi_{i,s}\) is the map that projects onto the \(i\)-th factor and then forgets all markings corresponding to nodes. Then \(\pi_{1,s}\) has relative dimension \(2\), since \(\M{g_1}{I_1+2}\) has two markings corresponding to nodes, while \(\pi_{2,s}\) and \(\pi_{3,s}\) have relative dimension \(1\). Therefore, \(e_L(\Delta_{\Gamma_1})=4\) (cf. \cref{lem:contracting subvariety}), since \(\kappa\) is ample. By symmetry, \(e_L(\Delta_{\Gamma_i})=4\) for every \(i\).

    Now, we prove the claim that the \(\Delta_{\Gamma_i}\) are precisely the codimension \(2\) subvarieties of \(\M{g}{n}\) such that \(e_L(Z)\ge 4\). Let \(Z\) be such a variety. In particular, \(e_L(Z)\ge 1\), so \(L|_Z\) is not big. Therefore, by \cref{thm:big} (1), \(Z\subseteq \Delta_{g_i, I_i}\) for some \(i\). Without loss of generality, assume that \(Z\subseteq \Delta_{g_1, I_1}\). Let 
    \[ \iota_2: \M{g_1}{I_1+1}\times \M{g_2+g_3}{I_2\sqcup I_3 +1}\to \M{g}{n} \]
    be the map corresponding to \(\Delta_{g_1, I_1}\). Let \(Z'\) be a subvariety of the source such that \(\iota_2(Z')=Z\). Since \(\iota_2\) is finite, \(e_{\iota_2^\ast L}(Z')\ge 4\). By direct computation,
    \[ \iota_2^\ast L = \pi_{1,s}^\ast \kappa + \pi_{2,s}^\ast (\kappa+\delta_{g_2, I_2})= (\pi_{1,s}\times \pi_{2,s})^\ast L'.   \]
    where \(L'\) is the corresponding divisor on \(\M{g_1}{I_1}\times \M{g_2+g_3}{I_2\sqcup I_3}\), given by the sum of \(\kappa\) and \(\kappa+\delta_{g_2, I_2}\) on the two factors. Again, \(\pi_{i,s}\) is the map that projects onto the \(i\)-th factor and then forgets all markings corresponding to nodes. Since \(\pi_{1,s}\) and \(\pi_{2,s}\) have relative dimension \(1\),
    \[ e_{L'}((\pi_{1,s}\times \pi_{2,s}) (Z'))\ge e_{\iota_2^\ast L}(Z')-2\ge 2. \]
    In particular, \((\pi_{1,s}\times \pi_{2,s}) (Z')\) is contained in the exceptional locus of \(L'\), which is contained in \(\M{g_1}{I_1}\times \Delta_{g_2, I_2}\). Therefore, \(Z'\) is contained in its inverse image, which is 
    \[ \left( \M{g_1}{I_1+1}\times \M{g_2}{I_2+2}\times \M{g_3}{I_3+1} \right)\cup \left( \M{g_1}{I_1+1}\times \M{g_2}{I_2+1}\times \M{g_3}{I_3+2} \right). \]
    Hence, since \(Z\) has codimension \(2\), it follows that \(Z\) is either \(\Delta_{\Gamma_2}\) or \(\Delta_{\Gamma_3}\). This proves the claim.

    Now, we prove that the \(\Delta_{\Gamma_i}\)'s are extremal. Without loss of generality, assume that \(i=1\). Moreover, assume that 
    \[ [\Delta_{\Gamma_1}]=\sum a_i [Z_i]  \]
    in \(\Eff^2(\M{g}{n})\), for some \(a_i>0\). Then, by \cref{lem:contracting subvariety} and the previous paragraph, the only possibility is
    \[ [\Delta_{\Gamma_1}]=b[\Delta_{\Gamma_2}]+c[\Delta_{\Gamma_3}]. \]
    after canceling \([\Delta_{\Gamma_1}]\) terms. Choose \(p\in I_3\). Then 
    \[ \Delta_{g_2, I_2\sqcup \{p\}}\cdot \Delta_{\Gamma_1}=\Delta_{g_2, I_2\sqcup \{p\}}\cdot \Delta_{\Gamma_2}=0, \Delta_{g_2, I_2\sqcup \{p\}}\cdot \Delta_{\Gamma_3}\ne 0 \]
    so \(c=0\). By the same argument using \(\Delta_{g_3, I_3\sqcup \{q\}}\) for \(q\in I_2\), \(b=0\), which is a contradiction. Therefore, \(\Delta_{\Gamma_1}\) spans an extremal ray.
\end{proof}

%% file: sections/inj.tex
This section proves the injectivity of certain pushforward maps on numerical Chow groups of moduli spaces of curves. These injectivity statements may be of independent interest. The first proposition is again from \cite{Bla22}, stated for the reader's convenience.

\begin{prop}(\cite[Proposition 4.2]{Bla22})\label{prop:genus0 pushforward injective}
    Let \(\iota: \M{0}{I+1}\times \M{g}{I^c+1} \to \M{g}{n} \) be the clutching map, where \( I\subseteq [n] \), \( |I|\ge 2 \). Then
    \[  \iota_\ast :\nc^d(\M{0}{I+1}\times \M{g}{I^c+1})\to \nc^{d+1}(\M{g}{n}) \]
    is injective for any \(d\).
\end{prop}

\begin{prop}\label{prop:N1 pushforward injective}
    Let \(\iota: \M{g_1}{I+1}\times \M{g_2}{I^c+1} \to \M{g}{n} \) be the clutching map, where \( g_1+g_2=g \). Let \(n_1=|I|, n_2=|I^c|\) and assume the following condition for \( \{i,j\}=\{1,2\} \):
    \[ \text{If }g_i\le g_j, \text{ then }n_i\ge 1. \]
    Then
    \[  \iota_\ast :\nc^1(\M{g_1}{I+1}\times \M{g_2}{I^c+1})\to \nc^2(\M{g}{n}) \]
    is injective.
\end{prop}

Note that the condition in \cref{prop:N1 pushforward injective} is exactly the same as that in \cref{prop:semigroup Kappa}. We write it in this form here because we work extensively with both \(\M{g_1}{I+1}\) and \(\M{g_2}{I^c+1}\), and so we present it symmetrically.

\begin{proof}
    By \cref{prop:genus0 pushforward injective}, we may assume that \(g_1, g_2\ge 1.\)

    As explained in \cite[proof of Lemma 7.4]{Cho25b}, using \cite[Theorem 0.1]{Mor01}, we obtain
    \[ \nc^1(\M{g_1}{I+1}\times \M{g_2}{I^c+1} )\simeq \nc^1(\M{g_1}{I+1})\times \nc^1(\M{g_2}{I^c+1}). \]
    By taking the dual of \(\iota_\ast\), we obtain
    \[ \iota^\ast : \nc_2(\M{g}{n})\to \nc_1(\M{g_1}{I+1})\times \nc_1(\M{g_2}{I^c+1})\]
    and it is enough to show that this is surjective.

    Note that for any \(g,n\), \( \nc_1(\M{g}{n}) \) is generated by classes of F-curves (see, e.g. \cite{GKM02}). Therefore, it suffices to produce a boundary \(2\)-stratum in \( \M{g}{n} \) whose pullback under \( \iota\) gives the desired F-curves. We will prove the following assertion: If \( n_2\ge 1 \) or \(g_2>g_1\), then for any F-curve \(F\) on \( \M{g_1}{n_1+1} \), there exists a boundary \(2\)-stratum \( S \subset \M{g}{n} \) such that
    \begin{equation}\label{eqn:fcurve}
        \iota^\ast [S] = c\cdot [F \times \mathrm{pt}]
    \end{equation}
    for some \(c>0\). By applying this to both \(\M{g_1}{I+1}\) and \(\M{g_2}{I^c+1}\), we obtain the theorem.

    The procedure is as follows. First, choose a representative of \(F\) exhibiting maximal degeneration. Here, maximal degeneration means being itself a boundary \(1\)-stratum, not merely numerically equivalent to one. Then, choose a maximally degenerated curve \(C_{\text{pt}}\), i.e. a boundary \(0\)-stratum, corresponding to \(\mathrm{pt}\). The image of \(F \times \mathrm{pt}\) under \(\iota\) corresponds to the boundary stratum obtained by attaching the \(1\)-stratum corresponding to \(F\) to the \(0\)-stratum corresponding to \(C_{\text{pt}}\). Next, smooth the attached node to construct a boundary \(2\)-stratum \(S\) in \(\M{g}{n}\). If we choose the representatives \(F\) and \(C_{\text{pt}}\) carefully, then \(S\) is not entirely contained in the image of \(\Delta_{g_1, I}\), so we can compute \(\iota^\ast [S]\) by intersecting \(S\) with \(\Delta_{g_1, I}\) up to multiplicity. This corresponds to the degenerations of \(S\) contained in \(\Delta_{g_1, I}\), and we can show that this always gives \([F \times \mathrm{pt}]\). 
    
    The following figure schematically illustrates the construction for \(\M{3}{3}\). 
    
    \begin{figure}[H]
    \centering
        \includegraphics[angle=0, width=0.8\textwidth]{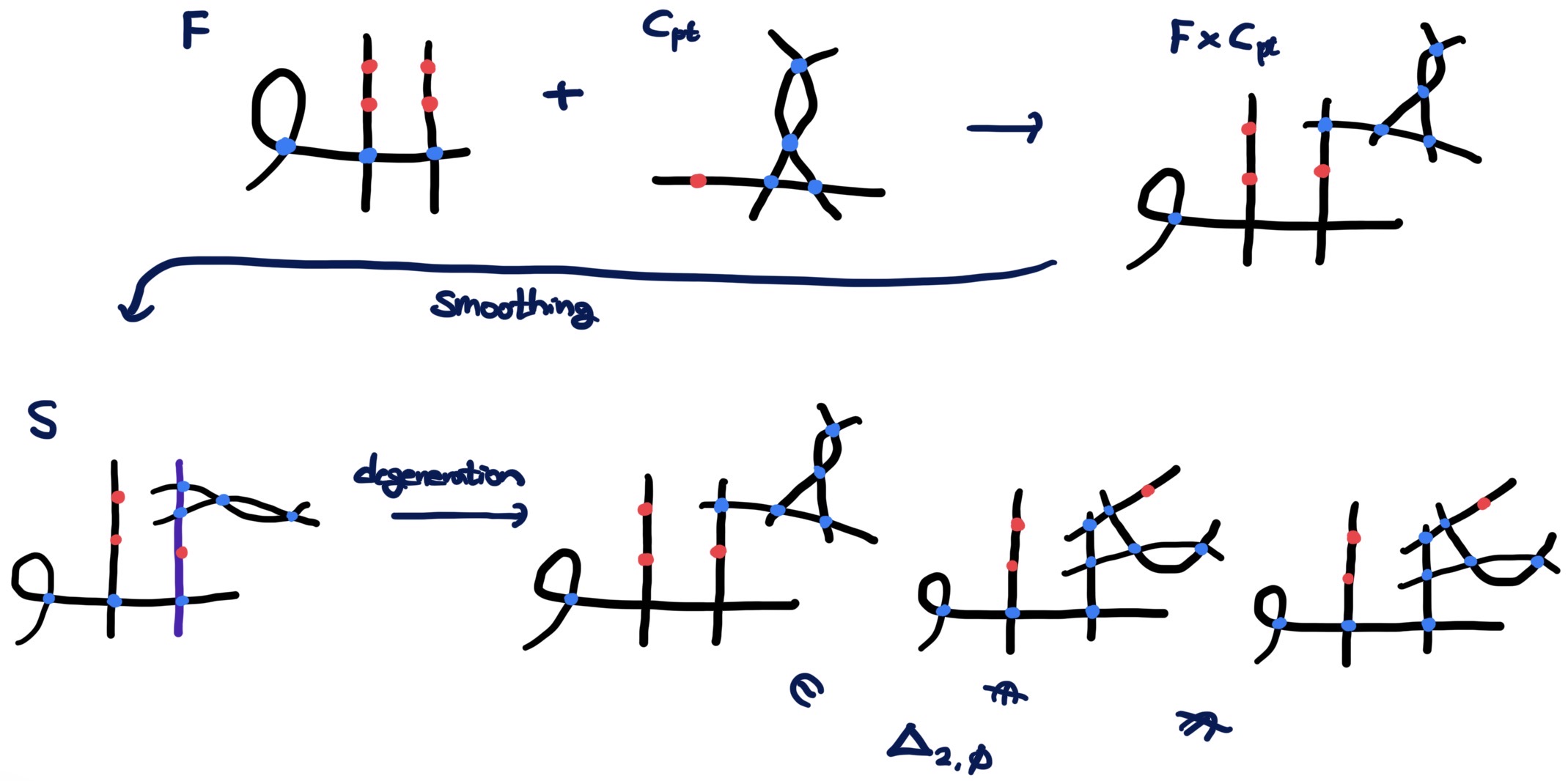}
        \caption{Diagrammatic description. Blue: nodes. Red: marked points.}
        \label{fig:1}
    \end{figure}

    Note that we draw only the degeneration of one of the genus \(0\) curves with \(4\) special points, since the degeneration of the other such curve cannot intersect \(\Delta_{2, \emptyset}\) because of the red dot (a marked point) on the right. This is a typical phenomenon: the only relevant degeneration occurs at the smoothed component.

    Strictly speaking, choosing a suitable representative of \(C_{\text{pt}}\) should not affect the conclusion, since all such representatives are numerically equivalent. However, some choices of \(C_{\text{pt}}\) may produce an \(S\) contained in \(\Delta_{g_1, I}\), so the pullback computation may involve excess intersection, adding further complications. Therefore, for the sake of a simpler proof, it is important to choose the representative carefully; this is ensured by the following lemma.

    \begin{lem}\label{lem:maximal degeneration}
        For \(g,n\ge 1\), there exists a boundary \(0\)-stratum of \(\M{g}{n}\) such that every separating node of the corresponding stable curve splits the curve into a genus \(0\) component that carries only markings from \([n-1]\) and a genus \(g\) component containing the \(n\)th marked point.
    \end{lem}

    \begin{proof}
        This is evident from the following figures.

        \begin{figure}[H]
        \centering
        \includegraphics[angle=0, width=0.6\textwidth]{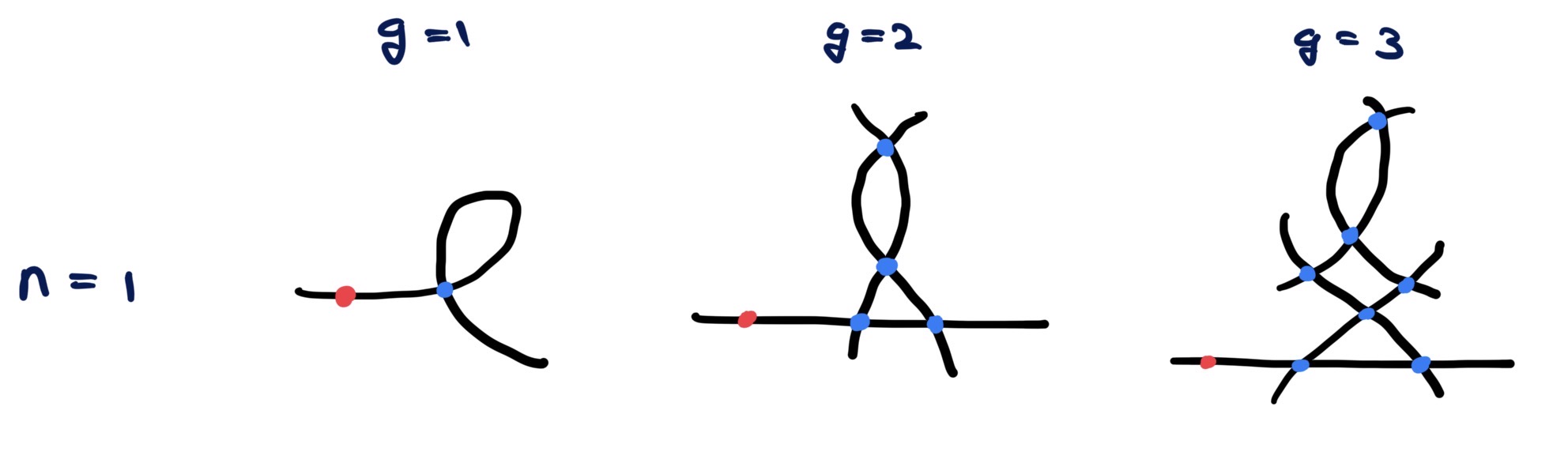}
        \caption{The case \(n=1\).}
        \end{figure}

        \begin{figure}[H]
        \centering
        \includegraphics[angle=0, width=0.6\textwidth]{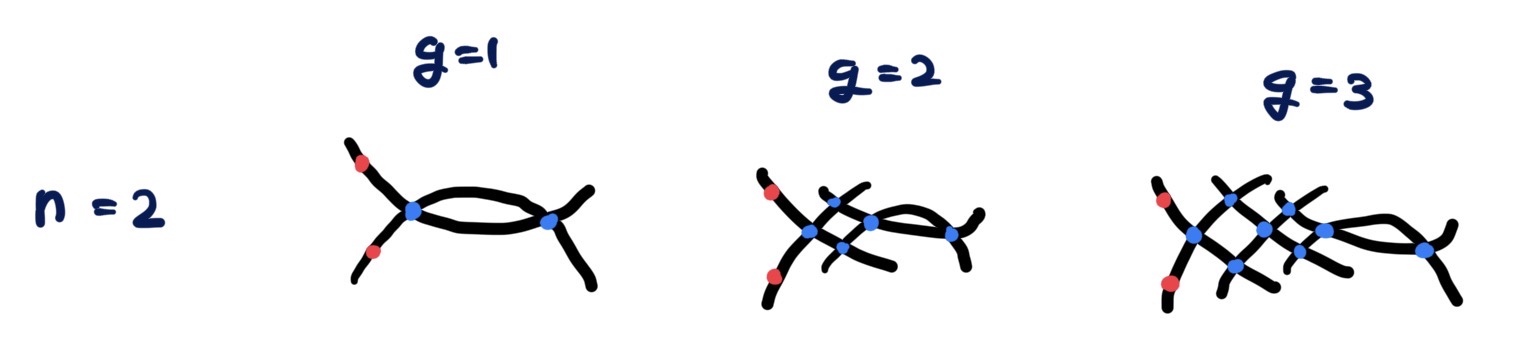}
        \caption{The case \(n=2\).}
        \end{figure}

        \begin{figure}[H]
        \centering
        \includegraphics[angle=0, width=0.6\textwidth]{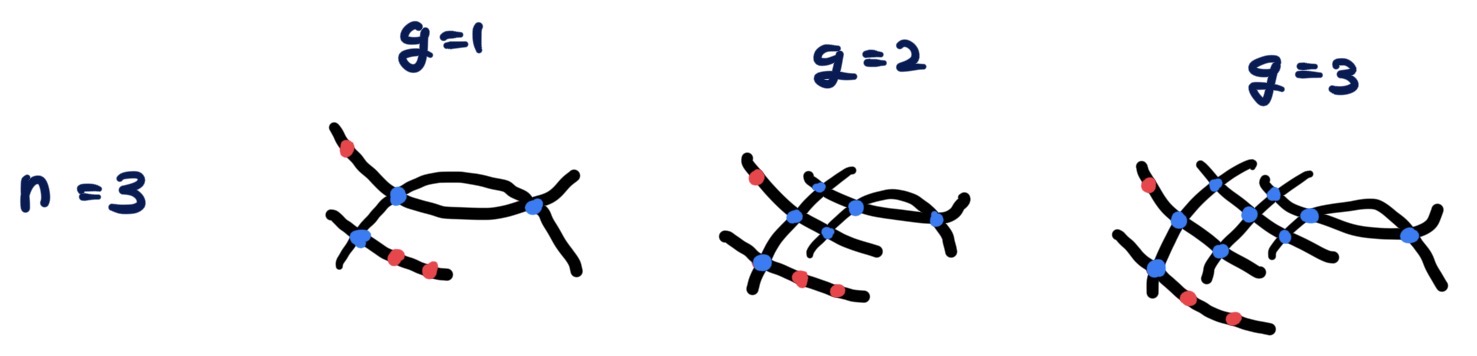}
        \caption{The case \(n=3\).}
        \label{fig:4}
        \end{figure}

        The general case is obtained simply by increasing the genus and adding more marked points. More precisely, to increase the genus, we can keep adding two components for each additional genus, as in the figures above. To increase the number of marked points by one, we attach an additional rational component carrying the new marking to the lower-left component of \cref{fig:4}, in the same pattern.
    \end{proof}

    We now give a rigorous justification of the procedure pictured in \cref{fig:1}. Choose any boundary \(1\)-stratum representing \(F\), and choose a boundary \(0\)-stratum \(C_{\text{pt}}\) as in \cref{lem:maximal degeneration}. Consider the boundary \(1\)-stratum \(\iota(F\times C_{\text{pt}})\), and obtain \(S\) by smoothing the attached node. This can be made more precise using the language of dual graphs. Let \(\Gamma_1\) be the dual graph of \(F\), \(\Gamma_2\) be the dual graph of \(C_{\text{pt}}\), and let \(p_1\) (resp. \(p_2\)) be the \((n_i+1)\)st point of \(\M{g_1}{I+1}\) (resp. \(\M{g_2}{I^c+1}\)), i.e. the attaching marked point. Then we can construct \(\Gamma_0\) as the dual graph obtained by attaching \(\Gamma_1\) and \(\Gamma_2\) at the legs \(p_1\) and \(p_2\). Let \(e_0\) be the edge of \(\Gamma_0\) joining \(p_1\) and \(p_2\), let \(\Gamma\) be the dual graph obtained by contracting \(e_0\), and let \(v\) denote the resulting vertex of the contraction. Then \(S\) is the boundary \(2\)-stratum corresponding to \(\Gamma\).

    First, we need to check that \(S\not\subseteq \Delta_{g_1, I}\). If \(S\subseteq \Delta_{g_1, I}\), then \(\Gamma\) must have an edge \(e\) that separates the curve into two parts corresponding to \(\M{g_1}{I+1}\) and \(\M{g_2}{I^c+1}\). By construction, any edge of \(\Gamma\) is either an edge of \(\Gamma_1\) or an edge of \(\Gamma_2\). If \(e\in E(\Gamma_2)\), then \(\Gamma-e\) consists of a subgraph of \(\Gamma_2\) and another subgraph containing \(\Gamma_1\). However, by \cref{lem:maximal degeneration}, that subgraph of \(\Gamma_2\) has genus \(0\) and carries only markings in \(I^c\), since the last marked point is the attaching node. Therefore, this is impossible. If \(e\in E(\Gamma_1)\), then \(\Gamma-e\) consists of a subgraph of \(\Gamma_1\) and another subgraph properly containing \(\Gamma_2\). Hence, the subgraph strictly containing \(\Gamma_2\) must correspond to \(\M{g_1}{I+1}\), but this is impossible since \(g_2>g_1\) or \(n_2\ge 1\).

    We now prove \cref{eqn:fcurve}, which amounts to proving that the only codimension \(1\) degeneration of \(S\) lying in \(\Delta_{g_1, I}\) is \(\Delta_{\Gamma_0}\). By the same argument as in the previous paragraph, it is clear that we must degenerate \(v\), and that the edge \(e\) arising from this degeneration should correspond to the separating node. This is the typical phenomenon mentioned right before \cref{fig:1}. 
    
    The vertex \(v\) has genus \(0\) or \(1\) by the classification of F-curves; see, for example, \cite[Section 2]{GKM02} or \cite[Section 2]{Cho25b}. If \(g(v)=1\), then the only possibility is \(\M{g_1}{I+1}=\M{1}{1}\), but this is impossible by the hypothesis. Therefore, we may assume that \(g(v)=0\). Then the degree of \(v\) is \(4\) or \(5\), and to degenerate it we need to split its \(4\) or \(5\) half-edges into two groups. Note that two of its half-edges come from \(\Gamma_2\), which we denote by \(h_1\) and \(h_2\). It is enough to show that the only possible splitting of the half-edges of \(v\) is into \(\{h_1, h_2\}\) and the remaining half-edges. By the construction of \(\Gamma_2\) in \cref{lem:maximal degeneration}, if \(h_1\) and \(h_2\) lie in different groups, then \(e\) is not even a separating node. Hence \(h_1\) and \(h_2\) must lie in the same group, which proves the degree \(4\) case. If the degree is \(5\), then there is the additional possibility that the group containing \(h_1\) and \(h_2\) also contains a third half-edge \(h\). However, in this case, if we let \(\Gamma'\) be the dual graph of the degeneration, then \(\Gamma'-e\) contains a dual graph strictly containing \(\Gamma_2\), and hence this corresponds to \(\M{g_1}{I+1}\). This is impossible, again, since \(g_2>g_1\) or \(n_2\ge 1\). This finishes the proof of \cref{eqn:fcurve}.
\end{proof}

\begin{rmk}\label{rmk:inj}
    \begin{enumerate}
    \item For \(\iota:\M{g}{1}\times \M{g}{1}\to \Mg{2g}\), \cref{prop:N1 pushforward injective} does not hold, since \(\iota\) factors through \( \left( \M{g}{1}\times \M{g}{1} \right)/\Z_2 \). Therefore, \cref{prop:N1 pushforward injective} does not cover the case of \( \iota:\M{i}{1}\times \M{g-i}{[n]+1}\to \M{g}{n}  \) when \(2i\le g\). The author does not know whether \cref{prop:N1 pushforward injective} also holds in this case.
    \item The injectivity of pushforward along \(\iota\) for higher-codimension cycles is the bottleneck of the method used in this paper. Improving this part would directly strengthen \cref{thm:main theorem 1} (1). For example, if the analogue of \cref{prop:N1 pushforward injective} held for codimension \(2\) classes on the source, then for \(k\ge 3\) we can improve the bound to \(n\ge 2k-3\).
    \item There are some obstructions to extending \cref{prop:N1 pushforward injective} to higher codimension. One of the most important is the existence of non-tautological algebraic classes \cite{GP03, vZel18, ACC+25}. However, since every even cohomology class of \(\M{1}{n}\) is tautological \cite{Pet14}, it might be possible to extend \cref{prop:N1 pushforward injective} to higher codimension to some extent.
    \end{enumerate}
\end{rmk}

%% file: sections/ext.tex
This section contains the proofs of the main theorems, \cref{thm:main theorem 1} and \cref{thm:main theorem 2}. These proofs use the following lemma.

\begin{lem}\label{lem:thm11}
    Let \(\iota:\M{i}{I+1}\times \M{g-i}{I^c+1}\to \M{g}{n}\) be a clutching map satisfying the condition of \cref{prop:semigroup Kappa} (or equivalently, the assumption of \cref{prop:N1 pushforward injective}), and let \(\pi_{1,s}\) and \(\pi_{2,s}\) be the projections defined in \cref{eqn:proj}. If \(Z\subseteq \M{i}{I+1}\times \M{g-i}{I^c+1}\) is a subvariety such that 
    \begin{enumerate}
        \item \(Z\) is a divisor on \(\M{i}{I+1}\times \M{g-i}{I^c+1}\) or \(i=0\), and
        \item \( [Z] \) spans an extremal ray of \(\Eff^k(\M{i}{I+1}\times \M{g-i}{I^c+1})\), and
        \item \( \dim \left(\pi_{1,s}\times \pi_{2,s} \right)(Z)<\dim Z \),
    \end{enumerate}
    then \([\iota(Z)]\) spans an extremal ray of \( \Eff^{k+1}( \M{g}{n} ) \).
\end{lem}

\begin{proof}
    We apply \cref{lem:main strategy} to \(\iota\) with \(k:=\) the codimension of \(Z\) in \(\M{i}{I+1}\times \M{g-i}{I^c+1}\). \cref{lem:main strategy} (1) follows from \cref{prop:genus0 pushforward injective} and \cref{prop:N1 pushforward injective}. We now verify condition (2) of \cref{lem:main strategy}. Combined with assumption (2) of the present lemma, this implies that \(\iota(Z)\) is extremal. Let
    \[  [\iota(Z)]=\sum_{j} a_j[Z_j] \]
    where \(a_j>0\) and \(Z_j\subseteq \M{g}{n}\) are subvarieties of \(\M{g}{n}\). We need to prove that \(Z_j\subseteq\Delta_{i, I}=\) image of \(\iota\). Let \(L=\kappa+\delta_{i, I}\). Then \(L\) is nef by \cref{prop:semigroup Kappa}, so we can apply \cref{lem:contracting subvariety}. By \cref{eqn:Kappa}, \( \iota^\ast L=\pi_{1,s}^\ast \kappa+\pi_{2,s}^\ast \kappa \). Hence
    \[ \iota^\ast L|_{Z}=(\pi_{1,s}^\ast \kappa+\pi_{2,s}^\ast \kappa)|_Z=(\pi_1^\ast\kappa+\pi_2^\ast \kappa)|_{\left(\pi_{1,s}\times \pi_{2,s} \right)(Z)} \]
    where \(\pi_1\) and \(\pi_2\) are the projections on \( \M{i}{I}\times \M{g-i}{I^c} \). By (3), this implies that \(L\) is not big on \(\iota(Z)\). Therefore, by the last assertion of \cref{lem:contracting subvariety}, the \(Z_j\) are contained in the exceptional locus \(\mathbb{E}(L)\) of \(L=\kappa+\delta_{i, I}\). Since \(\kappa\) is ample, \cref{thm:big} (1) implies that \(\mathbb{E}(L)\subseteq \Delta_{i,I}\). Hence \(Z_j\subseteq\Delta_{i, I}\), which completes the proof. 
\end{proof}

\begin{proof}[Proof of \cref{thm:main theorem 1} (1)]
    First, consider the case \(k=2\). Let \(D\) be an extremal effective divisor on \(\M{1}{n+1}\), where \(n\ge 2\). Then, by \cref{lem:extremal product everything}, \(Z_D:=D\times \M{g-1}{1}\) is an extremal effective divisor on \(\Delta_{1, [n]}\). Since \(g\ge 3\), we can apply \cref{lem:thm11}. Because of the factor \(\M{g-1}{1}\), we have \(\dim \left(\pi_{1,s}\times \pi_{2,s} \right)(Z_D)<\dim Z_D \). Therefore, by \cref{lem:thm11}, \([\iota(Z_D)]\) spans an extremal ray of \(\Eff^2(\M{g}{n})\). By \cite[Theorem 1.1]{CC14}, there are infinitely many extremal effective divisors on \(\M{1}{n+1}\) for \(n\ge 2\). Moreover, by \cref{prop:N1 pushforward injective}, the classes \([\iota(Z_D)]\) are distinct for distinct classes of \(D\). Therefore, \( \M{g}{n} \) has infinitely many extremal effective codimension \(2\)-cycles.

    Now we argue by induction. Assume that the statement holds for \(k\), and consider the case \(k+1\). The proof is similar to the case \(k=2\), but we use \(\iota:\M{g}{n-2}\times \M{0}{4}\to \M{g}{n}\) instead of \( \M{1}{n+1}\times \M{g-1}{1}\to \M{g}{n} \). For any \(n\ge 2k\), by the induction hypothesis, \(\M{g}{n-2}\) admits infinitely many extremal effective codimension \(k\)-cycles. Let \(D\) be any such cycle. Then, by \cref{lem:extremal product everything}, \(Z_D:= D\times \M{0}{4}\) is an extremal effective codimension \(k\)-cycle on \(\M{g}{n-2}\times \M{0}{4}\). Since the second factor has genus \(0\), we can apply \cref{lem:thm11}, and we obtain that \(\iota(Z_D)\) is an extremal effective codimension \(k+1\)-cycle on \(\M{g}{n}\). By \cref{prop:genus0 pushforward injective}, these cycles are distinct for distinct classes of \(D\). Hence, there are infinitely many extremal effective codimension \(k+1\)-cycles on \(\M{g}{n}\).
\end{proof}

\begin{rmk}\label{rmk:semiample}
    In the proof of \cref{thm:main theorem 1} (1), we used the nef divisor \(\kappa+\delta_{1, [n]}\). This illustrates why it is essential to extend \cite[Proposition 2.1]{CC15}, which is formulated in terms of morphisms, to \cref{lem:contracting subvariety}, which is formulated in terms of nef line bundles.

    By \cref{eqn:Kappa}, if \(\kappa+\delta_{1, [n]}\) is semiample, then the corresponding morphism \(f_1:\M{g}{n}\to X\) to a projective variety \(X\) is an isomorphism outside \(\Delta_{1, [n]}\) and agrees with \(\pi_{1,s}\times \pi_{2,s}\) on \(\Delta_{1, [n]}\). In \cref{prop:semiample}, we prove that such an \(f_1\) does not exist in characteristic \(0\). Therefore, the argument for \cref{thm:main theorem 1} (1) necessarily goes beyond \cite[Proposition 2.1]{CC15}.
\end{rmk}

\begin{prop}\label{prop:semiample}
    Assume that the base field has characteristic \(0\). Then the morphism \(f_1:\M{g}{n}\to X\) described in \cref{rmk:semiample} does not exist. 
\end{prop}

\begin{proof}
    Assume that such an \(f_1\) exists. Define another morphism \(f_2:=f_1\times \pi_{ [n-1]}\) from \(\M{g}{n}\). Then \(f_2\) is an isomorphism outside \(\Delta_{1, [n]}\) onto the image and agrees with \(\text{id}\times \pi_{2,s}\) on \(\Delta_{1, [n]}\). We will prove that such a morphism cannot exist.

    We use the F-curves of types 3 and 5, whose definitions are recalled in \cite[Section 2]{Cho25b}. Let \(H\) be an ample divisor on the image of \(f_2\), and let \(L:=f_2^\ast H\). Then \(L\) is semiample, and \(L\) contracts \(F_3^{1}([n])\) by the definition of \(f_2\). By \cite[proof of Corollary 1.2]{Cho25b}, \(L\) must also contract \(F_5^{g-1,0}(\emptyset, [n])\). However, it is evident that \(f_2\) does not contract \(F_5^{g-1,0}(\emptyset, [n])\). This is a contradiction.
\end{proof}

We now turn to the proof of \cref{thm:main theorem 1} (2). Here, we cannot directly apply \cref{lem:thm11} since we only know that \(\pEff^k(\M{0}{n})\) is not polyhedral, and do not know whether \(\Eff^k(\M{0}{n})\) has infinitely many extremal rays. Therefore, instead of proving a stronger statement that \(\Eff^k(\M{g}{n})\) has infinitely many extremal rays, we will prove a weaker one, namely, \(\Eff^k(\M{g}{n})\) is not polyhedral. We argue by contradiction, and then the assumption that \(\Eff^k(\M{g}{n})\) is polyhedral gives strong additional information. In particular, this implies that every pseudo-effective cycle is indeed effective. Using this and \cref{lem:contracting subvariety}, we will show that under this assumption, \(\pEff^{k-1}(\M{0}{n+1})=\Eff^{k-1}(\M{0}{n+1})\). Hence, we can run the argument using \cref{lem:thm11}.

\begin{proof}[Proof of \cref{thm:main theorem 1} (2)] Following the outline above, we divide the proof into three steps.

    \textbf{Step 1. }\(\pEff^k(\M{0}{n})\) is not polyhedral if \(n\ge k+7\) and \(k\ge 1\). 
    
    By \cite{CLTU23} and \cite{Mul25}, \(\pEff^1(\M{0}{n})\) is not polyhedral for \(n\ge 8\), so the case \(k=1\) follows. Assume that this holds for all smaller codimensions, and consider the case of codimension \(k\). By \cref{prop:numerical chow facts} (5), \(\pi_{n,\ast} \pEff^k(\M{0}{n})=\pEff^{k-1}(\M{0}{n-1})\). By the induction hypothesis, \(\pEff^{k-1}(\M{0}{n-1})\) is not polyhedral. Therefore, \(\pEff^k(\M{0}{n})\) is also not polyhedral. 

    \textbf{Step 2. } Assume that \(\Eff^k(\M{g}{n})\) is polyhedral. Then \(\pEff^{k-1}(\M{0}{n+1})=\Eff^{k-1}(\M{0}{n+1})\).

    If \(\Eff^k(\M{g}{n})\) is polyhedral, then every pseudo-effective codimension \(k\)-cycle on \(\M{g}{n}\) is an effective codimension \(k\)-cycle in \(N^k(\M{g}{n})_\R\). Let \(D\) be a pseudo-effective codimension \(k-1\) cycle on \(\M{0}{n+1}\). Consider the clutching map 
    \[ \iota:\M{0}{n+1}\times \M{g}{1}\to \M{g}{n}. \]
    Then \( \iota_\ast(D\times [\M{g}{1}]) \) is a pseudo-effective codimension \(k\)-cycle on \(\M{g}{n}\), hence an effective codimension \(k\)-cycle. Therefore, there exist codimension \(k\) subvarieties \(Z_1,\cdots, Z_m\subseteq \M{g}{n}\) such that
    \[ \iota_\ast(D\times [\M{g}{1}])=\sum_{j=1}^m a_j [Z_j] \text{ where }a_j>0. \]
    Let \(L=\kappa+\delta_{0, [n]}\), which is nef by \cref{prop:semigroup Kappa}, and let \(d\) be the dimension of \( D\times [\M{g}{1}] \). Then, by \cref{eqn:Kappa}, 
    \begin{align*}
         c_1(L)^d\cap \iota_\ast(D\times [\M{g}{1}])&=\iota_\ast \left(c_1(\iota^\ast L)^d\cap (D\times [\M{g}{1}])  \right)\\&=\iota_\ast \left(c_1((\pi_{1,s}\times\pi_{2,s})^\ast\kappa )^d\cap (D\times [\M{g}{1}])  \right)
    \end{align*}
    Note that 
    \begin{align*}
        (\pi_{1,s}\times\pi_{2,s})_\ast \left(c_1((\pi_{1,s}\times\pi_{2,s})^\ast\kappa )^d\cap (D\times [\M{g}{1}]) \right)& = c_1(\kappa)^d \cap (\pi_{1,s}\times\pi_{2,s})_\ast(D\times [\M{g}{1}])\\
        &= c_1(\kappa)^d \cap (\pi_{1,s,\ast}D\times \pi_{2,s, \ast}[\M{g}{1}])=0.
    \end{align*}
    since \(\dim \pi_{2,s}(\M{g}{1})<\dim \M{g}{1}\). Moreover, since \(d=\dim  D\times [\M{g}{1}] \), \(c_1((\pi_{1,s}\times\pi_{2,s})^\ast\kappa )^d\cap (D\times [\M{g}{1}])\in N_0(\Delta_{0, [n]})_\R\), hence can be identified with a real number, namely its degree. Therefore,
    \[c_1((\pi_{1,s}\times\pi_{2,s})^\ast\kappa )^d\cap (D\times [\M{g}{1}])=0\]
    and hence \(c_1(L)^d\cap \iota_\ast(D\times [\M{g}{1}])=0\). Therefore, \(e_L(\iota_\ast(D\times [\M{g}{1}]))\ge 1\), so by \cref{lem:contracting subvariety}, \(e_L(Z_j)\ge 1\). By the argument in the proof of \cref{lem:contracting subvariety}, \(L|_{Z_j} \) is not big, so each \(Z_j\) is contained in \(\Delta_{0, [n]}\). Therefore, \([Z_j]=\iota_\ast [Z_j']\) for the corresponding subvariety \(Z_j'\subseteq \Delta_{0, [n]}\), so by \cref{prop:genus0 pushforward injective},
    \begin{equation}\label{eqn:tlqkf}
        D\times [\M{g}{1}]=\sum_{j=1}^m a_j [Z_j'].
    \end{equation}
    We now prove that \(D\) is effective. Let \(H\) be an ample divisor on \(\M{g}{1}\). Then
    \[ c_1(\pi_2^\ast H)^{3g-2}\cap (D\times [\M{g}{1}])= \deg (H^{3g-2})\cdot (D\times \text{pt}). \]
    See, e.g. \cite[\href{https://stacks.math.columbia.edu/tag/0FBY}{Tag 0FBY}]{Stacks}. Since \(\pi_2^\ast H\) is semiample, by \cref{prop:numerical chow facts} (3), this is effective by \cref{eqn:tlqkf}. Pushing forward along \(\pi_1\), we conclude that \(\deg (H^{3g-2})\cdot D\) is effective. Since \(\deg H^{3g-2}>0\), it follows that \(D\) is effective. 

    \textbf{Step 3. }Proof of the theorem.

    Assume \(n\ge k+5\) and \(k\ge 2\). We need to prove that \(\Eff^k(\M{g}{n})\) is not polyhedral. Suppose that \(\Eff^k(\M{g}{n})\) is polyhedral. By the assumption, Step 1 and Step 2, \(\Eff^{k-1}(\M{0}{n+1})=\pEff^{k-1}(\M{0}{n+1})\) and it is not polyhedral. Since the cone is moreover closed and salient, the non-polyhedrality implies that it has infinitely many extremal rays. Let \(D\in \Eff^{k-1}(\M{0}{n+1})\) span an extremal ray. Then, \(D\times [\M{g}{1}]\) also spans an extremal ray of \(\Eff^{k-1}(\Delta_{0, [n]})\) by \cref{lem:extremal product everything}. Therefore, \(D\times [\M{g}{1}]\) satisfies all the conditions of \cref{lem:thm11} by \cref{prop:genus0 pushforward injective}, so \( \iota(D\times [\M{g}{1}]) \) is also extremal in \(\Eff^k(\M{g}{n})\). However, there are infinitely many such \(D\), and hence infinitely many such \( \iota(D\times [\M{g}{1}]) \) by \cref{prop:genus0 pushforward injective}. This contradicts our initial assumption that \(\Eff^k(\M{g}{n})\) is polyhedral. Therefore, \(\Eff^k(\M{g}{n})\) is not polyhedral.
\end{proof}

Now we prove \cref{thm:main theorem 2}. For this, we need a combinatorial lemma. Recall that a graph is called a \textbf{tree} if it is connected and contains no cycles, and a \textbf{leaf} of a tree is a vertex of degree \(1\). For any two vertices \(v,w\) of a connected graph \(G\), the \textbf{distance} \(d(v,w)\) between \(v\) and \(w\) is the length of a shortest path. For a stable graph \(\Gamma\), the underlying graph is the graph with vertex set \(V(\Gamma)\) and edge set \(E(\Gamma)\), disregarding the legs.

\begin{defn}\label{defn:rational tail}(\cite{Bla22})
    A stable graph \(\Gamma\) is called \textbf{of rational tails type} if its underlying graph is a tree and all of the genus is concentrated on a single vertex. The corresponding boundary stratum \(\Delta_{\Gamma}\) is called a \textbf{rational tails boundary stratum}.
\end{defn}

\begin{lem}\label{lem:combinatorics}
    Let \(\Gamma\) be a stable graph of rational tails type with at least two vertices and \(g\ge 1\). Then \(\Gamma\) either has
    \begin{enumerate}
        \item a genus \(0\) leaf \(v_0\) that is not trivalent, or
        \item an edge \(e\) such that one of the two components of \(\Gamma-e\) consists only of genus \(0\) trivalent vertices, and the other endpoint of \(e\) is non-trivalent or has genus \(g\).
    \end{enumerate}
\end{lem}

\begin{proof}
    Assume that \(\Gamma\) has no genus \(0\) non-trivalent leaf. Let \(v\in V(\Gamma)\) be the vertex with \(g(v)=g\). If every other vertex is trivalent, then any edge \(e\) incident to \(v\) satisfies the condition. Otherwise, let \(w\) be a non-trivalent genus \(0\) vertex with \(d(v,w)\) maximal. Since the underlying graph is a tree, there is a unique path \(P\) from \(v\) to \(w\). Because \(w\) is not a leaf, there is an edge \(e\) incident to \(w\) and not contained in \(P\). We claim that this \(e\) satisfies condition (2). The graph \(\Gamma-e\) consists of two components: one containing \(P\), \(v\), and \(w\), and the other containing the remaining vertices. If the latter component contained a non-trivalent vertex \(w'\), then the unique path from \(v\) to \(w'\) would pass through \(w\), since the underlying graph of \(\Gamma\) is a tree. This would imply \(d(v,w')>d(v,w)\), contradicting the maximality of \(d(v,w)\).
\end{proof}

\begin{proof}[Proof of \cref{thm:main theorem 2}]
    Since any boundary stratum on \(\M{0}{n}\) is extremal by \cite[Theorem 1.1]{Bla22}, we may assume that \(g\ge 1\). We use induction on the codimension of \(\Delta_{\Gamma}\), i.e. the number of edges \(|E(\Gamma)|\) of \(\Gamma\). If the codimension is \(0\), then \([\Delta_{\Gamma}]=[\M{g}{n}]\), so this is trivially extremal. 
    
    Assume that the statement holds for every codimension \(<k\), and let \(\Gamma\) be a dual graph of rational tails type of codimension \(k\). By \cref{lem:combinatorics}, \(\Gamma\) either has a non-trivalent leaf of genus \(0\) or an edge satisfying (2). First, assume that \(\Gamma\) has a genus \(0\) non-trivalent leaf \(v\). The proof in this case is the same as the proof of \cite[Corollary 4.9]{Bla22}. Let \(\Gamma':=\Gamma-v\) and let \(l\ge 3\) be the number of legs on \(v\). Then \(\Gamma'\) is of rational tails type with codimension \(k-1\). Therefore, \(\Delta_{\Gamma'}\) is extremal by the induction hypothesis. Note that under the clutching map
    \[ \iota:\M{0}{l+1}\times \M{g}{n-l+1}\to \M{g}{n},  \]
    we have \( \Delta_{\Gamma}=\iota(\M{0}{l+1}\times \Delta_{\Gamma'}) \). By \cite[Lemma 4.4]{Bla22} or by \cref{lem:extremal product everything} and \cref{lem:thm11}, \(\Delta_\Gamma\) is also extremal.

    Now assume that there is an edge \(e\) satisfying \cref{lem:combinatorics} (2). By splitting \(e\) in \(\Gamma\) into two legs \(l_1\) and \(l_2\), we obtain two dual graphs \(\Gamma_1\) and \(\Gamma_2\), where \(\Gamma_1\) has genus \(0\), \(\Delta_{\Gamma_1}\) has dimension \(0\), \(\Gamma_2\) has genus \(g\) and is of rational tails type. Let \(I\) be the set of markings of \(\Gamma\) contained in \(\Gamma_1\). Then, under the clutching map
    \[  \iota:\M{0}{I+1}\times \M{g}{I^c+1}\to \M{g}{n},  \]
    we have \( \Delta_{\Gamma}=\iota(\Delta_{\Gamma_1}\times \Delta_{\Gamma_2}) \). We now check that we can apply \cref{lem:thm11}. Condition (1) of \cref{lem:thm11} is trivial. By the induction hypothesis, \(\Delta_{\Gamma_2}\) is extremal. Since \(\Delta_{\Gamma_1}\) is a point, \cref{lem:extremal product point} shows that \(\Delta_{\Gamma_1}\times \Delta_{\Gamma_2}\) is extremal. Hence condition (2) of \cref{lem:thm11} holds. Note that the vertex of \(\Gamma_2\) containing \(l_2\) is either non-trivalent or of genus \(g\ge 1\). In both cases,
    \[ \dim \pi_{2,s}(\Delta_{\Gamma_2})<\dim \Delta_{\Gamma_2}. \]
    Thus condition (3) of \cref{lem:thm11} also holds. Therefore, \(\Delta_{\Gamma}=\iota(\Delta_{\Gamma_1}\times \Delta_{\Gamma_2})\) is extremal. This completes the induction and hence the proof.
\end{proof}